\documentclass[12pt, a4paper, leqno]{article}
\usepackage[margin=2.5cm]{geometry}
\usepackage[utf8]{inputenc} 
\usepackage{lmodern}
\usepackage[T1]{fontenc}
\usepackage[english]{babel}

\usepackage[runin]{abstract}
\usepackage{titling}

\usepackage{amsmath}
\usepackage{amsthm}
\usepackage{amsfonts}
\usepackage{amssymb}
\usepackage{mathtools}
\usepackage{clrscode}
\usepackage{enumitem}
\usepackage{mdwlist}

\abslabeldelim{.}

\newcommand{\keywordsname}{Key words}
\makeatletter
\newcommand{\keywords}[1]{%
\def\thekeywords{#1}%
\begin{@bstr@ctlist}
\hspace*{\abstitleskip}{\abstractnamefont\keywordsname\@bslabeldelim}\abstracttextfont\
#1%
\par\end{@bstr@ctlist}
}
\makeatother

\newcommand{\subjclassname}{Mathematics subject classification}
\makeatletter
\newcommand{\subjclass}[2][2020]{%
\begin{@bstr@ctlist}
\hspace*{\abstitleskip}{\abstractnamefont\subjclassname\ (#1)\@bslabeldelim}\abstracttextfont\
#2%
\par\end{@bstr@ctlist}
}
\makeatother

\makeatletter
\def\and{
	\end{tabular}%
	and%
	\begin{tabular}[t]{c}}%
\makeatother

\makeatletter
\def\thanks#1{
\protected@xdef\@thanks{\@thanks
\protect\footnotetext[\the\c@footnote]{#1}}%
}
\makeatother

\makeatletter
\let\addresses\@empty      
\newcommand{\address}[2][]{\g@addto@macro\addresses{\address{#1}{#2}}}
\newcommand{\curraddr}[2][]{\g@addto@macro\addresses{\curraddr{#1}{#2}}}
\newcommand{\email}[2][]{\g@addto@macro\addresses{\email{#1}{#2}}}
\newcommand{\urladdr}[2][]{\g@addto@macro\addresses{\urladdr{#1}{#2}}}
\def\enddoc@text{
  \ifx\@empty\addresses \else\@setaddresses\fi}
\AtEndDocument{\enddoc@text}
\def\emailaddrname{E-mail address}
\def\@setaddresses{\par
  \nobreak \begingroup
  \interlinepenalty\@M
  \def\address##1##2{\begingroup%
    \par\addvspace\bigskipamount
    \@ifnotempty{##1}{(\ignorespaces##1\unskip) }%
    {\noindent\ignorespaces##2}\par\endgroup}%
  \def\email##1##2{\begingroup
    \@ifnotempty{##2}{\nobreak\noindent\emailaddrname
      \@ifnotempty{##1}{, \ignorespaces##1\unskip}\/:\space
      \ttfamily##2\par}\endgroup}%
  \addresses
  \endgroup
}

\makeatother

\makeatletter
\def\cstar#1{\expandafter\@cstar\csname c@#1\endcsname}
\def\@cstar#1{\ifcase#1\or $\ast$\or $\ast\ast$\or $\ast\ast\ast$\fi}
\AddEnumerateCounter{\cstar}{\@cstar}{$\ast\ast\ast$}
\makeatother

\newlist{conditions}{enumerate}{1}
\newlist{iconditions}{enumerate}{1}
\newlist{inthm}{enumerate}{1}
\setlist[conditions]{label=\normalfont(\alph*),ref=\normalfont(\alph*)}
\setlist[iconditions]{label=\normalfont(\roman*),ref=\normalfont(\roman*)}
\setlist[inthm]{label=\normalfont(\thetheorem.\arabic*),ref=\normalfont(\thetheorem.\arabic*),leftmargin=*}

\newcommand{\PB}{\mathbb{P}}
\newcommand{\R}{\mathbb{R}}
\newcommand{\SB}{\mathbb{S}}

\newcommand{\Cinfty}{\mathcal{C}^{\infty}}
\newcommand{\GL}{\mathrm{GL}}
\newcommand{\ON}{\mathrm{O}}

\newtheorem{theorem}{Theorem}[section]
\newtheorem{proposition}[theorem]{Proposition}
\newtheorem{lemma}[theorem]{Lemma}
\theoremstyle{definition}
\newtheorem*{acknowledgements}{Acknowledgements}
\newtheorem{definition}[theorem]{Definition}
\newtheorem{example}[theorem]{Example}

\DeclarePairedDelimiter\abs{\lvert}{\rvert}%
\DeclarePairedDelimiter\norm{\lVert}{\rVert}%

\makeatletter
\let\oldabs\abs
\def\abs{\@ifstar{\oldabs}{\oldabs*}}
\let\oldnorm\norm
\def\norm{\@ifstar{\oldnorm}{\oldnorm*}}
\makeatother

\title{Dominating real algebraic morphisms}
\date{}
\author{Wojciech Kucharz}

\address{Wojciech Kucharz\\Institute of Mathematics\\Faculty of Mathematics and Computer
Science\\Jagiellonian University\\\L{}ojasiewicza 6\\30-348
Krak\'ow\\Poland}
\email{Wojciech.Kucharz@im.uj.edu.pl}

\usepackage[pdftex, pdfauthor={\theauthor}, pdftitle={\thetitle}]{hyperref}

\begin{document}
\maketitle
\thispagestyle{empty}

\begin{abstract}
Let $X$ and $Y$ be nonsingular real algebraic varieties, $\dim X \geq \dim Y$. Assume that the variety $Y$ is malleable, compact and connected. Our main result implies that each regular map from $X$ to $Y$ is homotopic to a surjective regular map. The class of malleable varieties includes all homogeneous spaces for linear real algebraic groups.
\end{abstract}

\hypersetup{pdfkeywords={\thekeywords}}

\section{Introduction}\label{sec:1}
Throughout this note we use the term \emph{real algebraic variety} to mean a ringed space with structure sheaf of $\R$-algebras of $\R$-valued functions, which is isomorphic to a Zariski locally closed subset of real projective $n$-space $\PB^n(\R)$, for some $n$, endowed with the Zariski topology and the sheaf of regular functions. This is compatible with \cite{bib1, bib8}, which contain a detailed exposition of real algebraic geometry. Recall that each real algebraic variety in the sense used here is actually \emph{affine}, that is, isomorphic to an algebraic subset of $\R^n$ for some~$n$, see \cite[Proposition~3.2.10 and Theorem~3.4.4]{bib1}. Morphisms of real algebraic varieties are called \emph{regular maps}. Each real algebraic variety is also equipped with the Euclidean topology determined by the usual metric on $\R$. Unless explicitly stated otherwise, all topological notions relating to real algebraic varieties refer to the Euclidean topology.

\begin{definition}\label{def-1-1}
Let $Y$ be a nonsingular real algebraic variety.
\begin{iconditions}
\item\label{def-1-1-i} A \emph{dominating map} for $Y$ at a point $y_0 \in Y$ is a regular map $f \colon \R^n \to Y$, for some nonnegative integer $n$, such that $f(0) = y_0$ and the derivative $d_0 f \colon T_0 \R^n \to T_{y_0} Y$ is surjective; if such a map $f$ exists, then $Y$ is said to be \emph{dominable at $y_0$}.

\item\label{def-1-1-ii} A \emph{dominating spray} for $Y$ is a regular map $s \colon Y \times \R^n \to Y$, for some nonnegative integer $n$, such that for every point $y \in Y$ the equality $s(y,0) = y$ holds, and the map $s(y, \cdot) \colon \R^n \to Y$, $v \mapsto s(y,v)$ is dominating for $Y$ at $y$.

\item\label{def-1-1-iii} The variety $Y$ is called \emph{malleable} if it admits a dominating spray.
\end{iconditions}
\end{definition}

Obviously, if $Y$ is dominable at $y_0\in Y$, then there exists a dominating map $\R^n \to Y$ for $Y$ at $y_0$, where $n$ is the dimension of the irreducible component of $Y$ containing $y_0$. As noted in Proposition~\ref{prop-1-4}, Definition~\ref{def-1-1}\ref{def-1-1-i} is closely related to the classical notion of unirational variety.

Every malleable variety is dominable at each of its points, but it is not known whether the converse is also true.

The notion of dominating spray introduced in \cite[Definition~2.1(ii)]{bib2} is more general than that in Definition~\ref{def-1-1}\ref{def-1-1-ii}. However, according to \cite[Lemma~2.2]{bib2}, the notion of malleable variety in \cite[Definition~2.1(iii)]{bib2} is identical with that in Definition~\ref{def-1-1}\ref{def-1-1-iii}. A large class of malleable real algebraic varieties was identified in \cite[Propositions 2.7~and~2.8]{bib2}, see also Example~\ref{ex-1-3} for a summary of these results. Moreover, \cite[Theorem~4.2]{bib2} on maps with values in malleable varieties plays a key role in all the main results of \cite{bib2}.

The aim of this note is to prove the following dominability result.

\begin{theorem}\label{th-1-2}
Let $Y$ be a nonsingular real algebraic variety. Assume that $Y$ is malleable, compact and connected. Then, for every nonsingular real algebraic variety $X$, with $\dim X \geq \dim Y$, and every regular map $f \colon X \to Y$, there exists a surjective regular map $g \colon X \to Y$ having the following two properties:
\begin{inthm}
\item\label{th-1-2-1} there exists a regular map $F \colon X \times \R \to Y$ such that $F(x,0)=f(x)$ and $F(x,1)=g(x)$ for all $x \in X$ (in particular, $f$ is homotopic to $g$), and

\item\label{th-1-2-2} for every point $y \in Y$ there exists a point $x \in X$ such that $g(x)=y$ and the derivative $d_x g \colon T_x X \to T_y Y$ is surjective.
\end{inthm}
\end{theorem}

The proof of Theorem~\ref{th-1-2} is deferred to Section~\ref{sec:2}. It is not clear whether the assumptions are optimal. Is malleability of $Y$ necessary? Is compactness of $Y$ necessary if $X$ is not compact?

If there exists a surjective regular map $\varphi \colon \R^n \to Y$, then $Y$ is connected. Moreover, for some point $u \in \R^n$ the derivative $d_u \varphi \colon T_u \R^n \to T_{\varphi(u)} Y$ is surjective, so the map $\R^n \to Y$, $x \mapsto \varphi(u+x)$ is dominating for $Y$ at $\varphi(u)$, and the variety $Y$ is dominable at $\varphi(u)$.

The following example illustrates the scope of applicability of Theorem~\ref{th-1-2}.

\begin{example}\label{ex-1-3}
Let $G$ be a linear real algebraic group, that is, a Zariski closed subgroup of the general linear group $\GL_n(\R)$, for some $n$. A \emph{$G$-space} is a real algebraic variety $Y$ on which $G$ acts, the action $G \times Y \to Y$, $(a,y) \mapsto a \cdot y$ being a regular map. We say that a $G$-space $Y$ is \emph{good} if $Y$ is nonsingular and for every point $y \in Y$ the derivative of the map $G \to Y$, $a \mapsto a \cdot y$ at the identity element of $G$ is surjective. Clearly, if $Y$ is homogeneous for $G$ (that is, $G$ acts transitively on $Y$), then $Y$ is a good $G$-space. By \cite[Proposition~2.8]{bib2}, every good $G$-space is malleable.

In particular, the unit $n$-sphere
\begin{equation*}
    \SB^n \coloneqq \{(x_0, \ldots, x_n) \in \R^{n+1} : 
    x_0^2 + \cdots + x_n^2 = 1\}
\end{equation*}
and real projective $n$-space $\PB^n(\R)$ are malleable varieties because they are homogeneous spaces for the orthogonal group $\ON(n+1) \subset \GL_{n+1}(\R)$.
\end{example}

Recall that an irreducible real algebraic variety $Y$ is called \emph{unirational} if there exists a regular map $\varphi \colon U \to Y$ from a Zariski open subset $U$ of $\R^n$, for some nonnegative integer~$n$, into $Y$ such that $\varphi(U)$ is Zariski dense in $Y$.

\begin{proposition}\label{prop-1-4}
If $Y$ is an irreducible nonsingular real algebraic variety, then the following conditions are equivalent:
\begin{conditions}
\item\label{prop-1-4-a} $Y$ is dominable at some point $y_0 \in Y$.

\item\label{prop-1-4-b} $Y$ is unirational.
\end{conditions}
\end{proposition}

\begin{proof}
\ref{prop-1-4-a}$\Rightarrow$\ref{prop-1-4-b}. Let $f \colon \R^n \to Y$ be a dominating map for $Y$ at $y_0$. Since the derivative $d_0 f \colon T_0 \R^n \to T_{y_0} Y$ is surjective, there exists a Euclidean open neighborhood $N$ of $0 \in \R^n$ such that $f(N)$ is a Euclidean open neighborhood of $y_0 \in Y$. It follows that the set $f(N)$ is Zariski dense in $\R^n$, so the set $f(\R^n)$ is also Zariski dense in $Y$.

\ref{prop-1-4-b}$\Rightarrow$\ref{prop-1-4-a}. Let $U$ be a Zariski open subset of $\R^n$, and $\varphi \colon U \to Y$ a regular map with $\varphi(U)$ Zariski dense in $Y$. We can choose a point $u \in U$ such that the derivative $d_u \varphi \colon T_u U \to T_{\varphi(u)} Y$ is surjective. Using a translation in $\R^n$, we may assume that ${u = 0 \in U}$. If a constant $c > 0$ is sufficiently small, then the regular map
\begin{equation*}
    h \colon \R^n \to \R^n, \quad x \mapsto \frac{cx}{1 + \norm{x}^2}
\end{equation*}
satisfies $h(\R^n) \subset U$. Since the derivative $d_0 h$ is an isomorphism, the regular map
\begin{equation*}
    g \colon \R^n \to Y, \quad x \mapsto \varphi(h(x))
\end{equation*}
is dominating for $Y$ at the point $\varphi(0)$.
\end{proof}

This note is based on our joint paper with Bochnak \cite{bib2} and is inspired by achievements in complex geometry, primarily the paper of Forstneri\v{c} \cite{bib3}, and in a general sense the celebrated article of Gromov \cite{bib5} and the monograph of Forstneri\v{c} \cite{bib4}.

\section{Some properties of maps into malleable varieties}\label{sec:2}

The following variant of \cite[Lemma~3.8]{bib2} will play a key role.

\begin{lemma}\label{lem-2-1}
Let $X$ be a nonsingular real algebraic variety, $Y$ a malleable nonsingular real algebraic variety, $U$ an open subset of $X$, and $\Phi \colon U \times [0,1] \to Y$ a continuous map such that for every $t \in [0,1]$ the map $\Phi_t \colon U 
\to Y$, $\Phi_t(x) = \Phi(x,t)$ is of class $\Cinfty$. Let $U_0$ be an open subset of $X$ whose closure $\overline U_0$ is compact and contained in $U$. Then there exist a dominating spray $s \colon Y \times \R^n \to Y$ for $Y$ and a continuous map $\eta \colon U_0 \times [0,1] \to \R^n$ such that
\begin{inthm}
\item\label{lem-2-1-1} $\eta(x,0) = 0$ for all $x \in U_0$,

\item\label{lem-2-1-2} $s(\Phi(x,0), \eta(x,t)) = \Phi(x,t)$ for all $(x,t) \in U_0 \times [0,1]$,

\item\label{lem-2-1-3} for every $t \in [0,1]$ the map $\eta_t \colon U_0 \to \R^n$, $\eta_t(x) = \eta(x,t)$ is of class $\Cinfty$.
\end{inthm}
\end{lemma}

\begin{proof}
Choose a dominating spray $s 
\colon Y \times \R^n \to Y$ for $Y$, as in Definition~\ref{def-1-1}\ref{def-1-1-ii}. Let $\tilde p \colon \tilde E \coloneqq (X \times Y) \times \R^n \to X \times Y$ be the product vector bundle and let $\tilde s \colon 
\tilde E \to X \times Y$ be defined by $\tilde s((x,y), v) = s(y,v)$. According to \cite[Lemma~4.1 and its proof]{bib2}, the triple $(\tilde E, \tilde p, \tilde s)$ is a dominating spray for the canonical projection $h \colon X \times Y \to X$, as in \cite[Definition~3.2(ii)]{bib2}.

Note that the continuous map
\begin{equation*}
    F \colon U \times [0,1] \to X \times Y, \quad F(x,t) = (x,(\Phi(x,t))
\end{equation*}
is a homotopy of $\Cinfty$ sections of $h \colon X \times Y \to X$, that is, for every $t \in [0,1]$ the $\Cinfty$ map $F_t \colon U \to X \times Y$, $F_t(x) = F(x,t)$ satisfies $h(F_t(x))=x$. Thus, by \cite[Lemma~3.8]{bib2}, the dominating spray $s \colon Y \times \R^n \to Y$ above can be chosen in such a way that there exists a continuous map $\eta \colon U_0 \times [0,1] \to \R^n$ satisfying \ref{lem-2-1-1} and \ref{lem-2-1-3}, and
\begin{equation*}
    \tilde s(F(x,0), \eta(x,t)) = F(x,t) \quad
    \text{for all } (x,t) \in U_0 \times [0,1].
\end{equation*}
The last equality implies directly that \ref{lem-2-1-2} holds.
\end{proof}

We are now ready to carry out our main task.

\begin{proof}[Proof of Theorem~\ref{th-1-2}]
The variety $Y$ is pure dimensional because it is nonsingular and connected. We may assume without loss of generality that the variety $X$ is pure dimensional. Set $n \coloneqq \dim X$ and $p \coloneqq \dim Y$.

Since $Y$ is compact, there exists a finite open cover $\{W_1, \ldots, W_r\}$ of $Y$, where each $W_i$ is $\Cinfty$ diffeomorphic to the open cube $(0,1)^p$ in $\R^p$. Moreover, there exist two collections $\{L_1, \ldots, L_r\}$ and $\{L'_1, \ldots, L'_r\}$ of compact subsets of $Y$ such that
\begin{equation*}
    Y = L'_1 \cup \cdots \cup L'_r
    \quad \text{and} \quad
    L'_i \subset \mathrm{interior}(L_i) \subset W_i \quad \text{for } 1 \leq i \leq r.
\end{equation*}
Let $\{V_1, \ldots, V_r\}$ be a collection of pairwise disjoint open subsets of $X$, where each $V_i$ is $\Cinfty$~diffeomorphic to the open cube $(0,1)^n$ in $\R^n$. We can choose a surjective $\Cinfty$ submersion $\varphi_i \colon V_i \to W_i$ and a compact subset $K_i$ of $V_i$ such that $\varphi_i(K_i) = L_i$, $1 \leq i \leq r$. Note that the map
\begin{equation*}
    \varphi \colon U \coloneqq V_1 \cup \cdots \cup V_r \to Y, \quad \varphi(x)=\varphi_i(x) \quad \text{for } 
    x \in V_i,\ 1 \leq i \leq r
\end{equation*}
is a $\Cinfty$ submersion satisfying $\varphi(K) =Y$, where $K \coloneqq K_1 \cup \cdots \cup K_r$.

As a side note, consider a continuous map $h \colon U \to Y$. Since each set $V_i$ is contractible, the restriction $h|_{V_i} \colon V_i \to Y$ is null homotopic. Consequently, the map $h \colon U \to Y$ is null homotopic because $Y$ is path connected.

It follows from the side note above that the $\Cinfty$ maps $f|_U$, $\varphi \colon U \to Y$ are homotopic. Thus there exists a $\Cinfty$ map $\Phi \colon U \times [0,1] \to Y$ such that $\Phi(x,0) = f(x)$ and $\Phi(x,1) = \varphi(x)$ for all $x \in U$, see for example \cite[Proposition~10.22]{bib7}. Choose an open subset $U_0$ of $X$ such that $K \subset U_0$ and the closure $\overline U_0$ is a compact subset of $U$. Let $s \colon Y \times \R^n \to Y$ be a dominating spray for $Y$ and let $\eta \colon U_0 \times [0,1] \to \R^n$ be a continuous map satisfying conditions \ref{lem-2-1-1}--\ref{lem-2-1-3} in Lemma~\ref{lem-2-1}. In particular,
\begin{equation*}
    s(f(x), \eta_1(x)) =
    s(\Phi(x,0), \eta_1(x,0)) = \Phi(x,1) = \varphi(x) \quad \text{for all } x \in U_0.
\end{equation*}
By the Stone--Weierstrass theorem, we can choose a regular map $\beta \colon X \to \R^n$ whose restriction $\beta|_{U_0}$ is arbitrarily close to $\eta_1$ in the weak $\Cinfty$ topology (defined in \cite[p.~36]{bib6}). The map
\begin{equation*}
    g \colon X \to Y, \quad
    g(x) = s(f(x), \beta(x))
\end{equation*}
is regular, and its restriction $g|_{U_0}$ is as close to $\varphi|_{U_0}$ in the weak $\Cinfty$ topology as we want if $\beta|_{U_0}$ is sufficiently close to $\eta_1$. We may assume that $g(K_i) \supseteq L'_i$ for $1 \leq i \leq r$, so $g(K) = Y$, and for every point $x \in K$ the derivative $d_x g \colon T_x X \to T_{g(x)} Y$ is surjective. In summary, $g \colon X \to Y$ is a surjective regular map satisfying condition \ref{th-1-2-2}. Moreover,
\begin{equation*}
    F \colon X \times \R \to Y, \quad
    F(x,t) = s(f(x), t\beta(x))
\end{equation*}
is a regular map satisfying condition \ref{th-1-2-1}.
\end{proof}

\begin{acknowledgements}
The author was partially supported by the National Science Center (Poland) under grant number 2018/31/B/ST1/01059.
\end{acknowledgements}

\phantomsection


\addcontentsline{toc}{section}{\refname}

\begin{thebibliography}{99}
\bibitem{bib1} J.~Bochnak, M.~Coste and M.-F.~Roy, Real Algebraic Geometry, Ergeb. Math. Grenzgeb.~3, vol.~36, Springer, Berlin (1998).

\bibitem{bib2} J.~Bochnak and W.~Kucharz, On approximation of maps into real algebraic homogeneous spaces, J.~Math. Pures Appl. 161 (2022), 111--134.

\bibitem{bib3} F.~Forstneri\v{c}, Surjective holomorphic maps onto Oka manifolds, in: Complex and Symplectic Geometry, vol.~21 of Springer INdAM Ser., 73--84, Springer, Cham (2017).

\bibitem{bib4} F.~Forstneri\v{c}, Stein Manifolds and Holomorphic Mappings, The Homotopy Principle in Complex Analysis, second edition, Ergeb. Math. Grenzgeb.~3, vol.~56, Springer, Cham (2017).

\bibitem{bib5} M.~Gromov, Oka's principle for holomorphic sections of elliptic bundles, J.~Amer. Math. Soc. 2 (1989), 851--897.

\bibitem{bib6} M.~W.~Hirsch, Differential Topology, GTM, vol.~33, Springer, New York (1997).

\bibitem{bib7} J.~M.~Lee, Introduction to Smooth Manifolds, Springer, New York (2003).

\bibitem{bib8} F.~Mangolte, Real Algebraic Varieties, Springer Monographs in Mathematics, Springer International Publishing (2020).
\end{thebibliography}
\end{document}